\documentclass[final,leqno]{siamltex}

\usepackage{amssymb, amsmath}

\newcommand{\baseRing}[1]{\ensuremath{\mathbb{#1}}}

\newcommand{\R}{\baseRing{R}}

\newcommand{\N}{\baseRing{N}}

\renewcommand{\phi}{\varphi}

\numberwithin{equation}{section}

\title{School Choice as a One-Sided Matching Problem:\\ Cardinal Utilities and Optimization\thanks{Aksoy, Azzam, Coppersmith and Karaali were partially supported by National Science Foundation Grant DMS-0755540. Karaali was partially supported by a Pomona College Hirsch Research Initiation Grant and a National Security Agency Young Investigator Award (NSA Grant $\#$H98230-11-1-0186). Zhao was partially supported by the Hutchcroft Fund of the Department of Mathematics and Statistics at Mount Holyoke College. Zhu was partially supported by a Mount Holyoke College Ellen P. Reese Fellowship.}}

\author{S. Aksoy\thanks{University of California, San Diego, La Jolla, CA, USA} \and
A. Azzam\thanks{University of California, Los Angeles, Los Angeles, CA, USA} \and
C. Coppersmith\thanks{Bryn Mawr College, Bryn Mawr, PA, USA} \and
J. Glass\thanks{University of North Texas, Denton, TX, USA} \and
G. Karaali\thanks{Pomona College, Claremont CA, USA} \and
X. Zhao\thanks{Mount Holyoke College, South Hadley, MA, USA} \and
X. Zhu\thanks{Mount Holyoke College, South Hadley, MA, USA}
}

\begin{document}

\maketitle

\begin{abstract}
The school choice problem concerns the design and implementation of matching mechanisms that produce school assignments for students within a given public school district. Previously considered criteria for evaluating proposed mechanisms such as stability, strategyproofness and Pareto efficiency do not always translate into desirable student assignments. In this note, we explore a class of one-sided, cardinal utility maximizing matching mechanisms focused exclusively on student preferences. We adapt a well-known combinatorial optimization technique (the Hungarian algorithm) as the kernel of this class of matching mechanisms. We find that, while such mechanisms can be adapted to meet desirable criteria not met by any previously employed mechanism in the school choice literature, they are not strategyproof. We discuss the practical implications and limitations of our approach at the end of the article.
\end{abstract}

\begin{keywords} 
assignment, matching, school choice, Hungarian algorithm
\end{keywords}

\begin{AMS}
90B80, 90C27, 91B14, 91B68
\end{AMS}

\pagestyle{myheadings}
\thispagestyle{plain}
\markboth{Aksoy {\it et al}.}{School Choice as a One-Sided Matching Problem}

\section{Introduction}
\label{S:Introduction}

School choice policies are processes by which families have some say in determining where their children go to school.  Since the late eighties such policies have been adopted by many school districts across the nation.  Before school choice, students were typically assigned to public schools according to proximity.  Since wealthy families have the means to move to areas with desirable or reputable schools, such families have always had \emph{de facto} school choice.  Children in families that could not afford such a privilege were left with no other option than to attend the closest school - whether or not the school was desirable and/or was a good fit.  Thus school choice has been celebrated as a successful tool giving more families the power to shape their children's education, regardless of socioeconomic background.

In many school districts where funding and experienced teachers are lacking, school quality is uneven, and often a small number of schools are strongly preferred over others. Since it is not possible to assign all students to their top choice school, the question of \emph{how} to assign students to schools is often regarded as the central issue in school choice. In order to safeguard parents who seek to have their children attend schools conveniently within walking distance, at which a sibling is enrolled, or those offering need-based programs, districts define and adhere to a handful of school {\it priorities} which encapsulate such constraints. Thus school choice can be viewed as a
two-sided matching problem. An extensive study of two-sided matching problems can be found in \cite{RoSo90}; a more recent historical overview is \cite{Ro08}.

Previous work on school choice as a matching problem evaluates assignments using the notions of stability, Pareto efficiency and strategyproofness. Though all worthy considerations, these do not necessarily suffice to promote the most desirable outcomes.
In the context of school choice, stability corresponds to preventing priority violations. A priority violation occurs when a student desires a school more than the school to which she was assigned, and has higher priority than a student assigned to her desired school. Preventing priority violations is desirable for a very pragmatic reason:
Students whose priorities are violated may have legitimate grounds for legal action. Even without legal recourse, it is often felt that students are ``entitled" to schools in which they have been prioritized.
However the focus on avoiding priority violations in current school choice mechanisms leads to documented inefficiencies.
See \cite{AbPaRo09}, \cite{ErEr08}, \cite{Ke10}, \cite{Ro82} for more on this potential tradeoff between stability and efficiency.

In this note, we explore a class of one-sided mechanisms that aim to best honor student preferences rather than focus on school priorities.\footnote{Two-sided matching problems where the preferences are one-sided have been considered in other contexts as well.  See for instance \cite{GGKKMM10}, a recent article on assigning papers to referees. In this regard we are not treading totally uncharted territory, but such an approach has not yet been attempted specifically for the SCP. In fact the most significant novelty in our approach is perhaps in devising mechanisms to maximize the total utility for students, without systematically considering the priority structures of the schools involved. Thus we propose, in this paper, ways to incorporate information about cardinal preferences into practically useful assignment mechanisms (cf.\ \S\ref{S:NewCriteria}).} In cities without well-defined or legally required priorities (e.g. those that use whole-city lotteries), such an approach might be considered by policy makers in an attempt to make a student-optimal matching. Even cities committed to respecting student priorities may find these ideas valuable as priorities may indeed be incorporated at an intermediate or a final stage, see the relevant discussion in \S\ref{SS:ImplementationMultMin}. On a more theoretical level, we believe that investigating the possible application of a well-known combinatorial optimization algorithm to the school choice problem is of value in itself.

These mechanisms work under a given choice of cardinal utility transformation - in other words, the mechanism designer cardinalizes ordinal preferences in a way that respects the ordering. After students are matched to schools, their total cardinal utility assigns a numerical ``cost" to each matching, and so we conceptualize the school choice problem as a ``cost-minimizing" assignment problem. We show how a well-known optimization algorithm - the Hungarian algorithm - can be adapted to find ``cost-minimizing" assignments with respect to a given choice of cardinal utility transformation. While there are infinitely many such cardinal utility transformations, we illustrate the application of our mechanism by considering two: one which assumes uniform utility gaps and another which weights ordinal preferences exponentially so that the student receiving their least preferred school receives as preferred an assignment as possible. We show how both transformations reflect different economic theories of fairness; however, we do not argue in favor of any particular cardinal utility transformation over another, leaving such considerations to the reader.

We summarize some relevant recent work on school choice in \S\S\ref{SS:LitReview}. In \S\S\ref{SS:Notation} we introduce the notation and standard terminology used throughout the rest of the paper and simultaneously describe our model. In \S\ref{S:NewCriteria} we define cardinal utility transformations (\S\S\ref{SS:CUT}) and introduce two evaluation criteria that correspond to distinct choices of cardinal utility transformations (\S\S\ref{SS:Index}, \S\S\ref{SS:Rank}). 
We introduce our mechanisms in \S\ref{S:HA}, first providing an elementary description of the standard algorithm  (\S\S\ref{SS:Description}) and then explaining how we adapt it to the school choice problem (\S\S\ref{SS:Modifications}). We study various properties of our mechanisms (\S\S\ref{SS:Properties}, \S\S\ref{SS:HASA}) and discuss some implementation issues (\S\S\ref{SS:ImplementationMultMin}).
\S\ref{S:Conclusion} concludes this note with a discussion of its implications and a view toward future work.


\subsection{Research background}
\label{SS:LitReview}

School district policy decisions have long provided active lines of inquiry for public policy designers, operations researchers, economists and education administrators. Much of the relevant work has focused on designing school district boundaries in order to optimize various measures. For a diverse yet representative selection of work in this vein, see \cite{BrKn05}, \cite{CaShGuWe04}, \cite{FeGu90}, \cite{FrKo73}.

In our work we focus on assignment policy as a mechanism design problem, which provides a natural framework to investigate means of implementing social goals (cf.\ \cite{Ma08}). In the current school choice literature, there has been much work surrounding three specific mechanisms. The first two were introduced in \cite{AbSo03} while the third was presented in \cite{Ke10}.
\begin{enumerate}
\item Student-Optimal Stable Matching Mechanism (SOSM)
\item Top Trading Cycles Mechanism (TTC)
\item Efficiency Adjusted Deferred Acceptance Mechanism (EADAM)
\end{enumerate}

SOSM adapts the famous Gale-Shapley Deferred Acceptance (DA) algorithm \cite{GaSh62} to the school choice problem. It is well-established as a stable and strategyproof mechanism that has already been implemented in several large urban school districts \cite{AbPaRo09}, \cite{AbPaRoSo06}. However, when applied to large-scale data SOSM may lead to some welfare losses \cite{Ke10}. TTC is an alternative mechanism which promotes efficiency as opposed to stability, and is also strategyproof. The basic algorithm is to create trading cycles alternating between students and schools and to allow efficient matchings.
EADAM is proposed in \cite{Ke10} as a way to alleviate some of the efficiency costs of stability by iteratively running SOSM and modifying the preferences of any interrupters (i.e., students who cause others to be rejected from a school which later on rejects them) such that the SOSM outcome is Pareto dominated. As any Pareto domination of SOSM will lead to priority violations (cf.\ \cite{GaSh62}), EADAM leads to at least one priority violation. We will not need the specific processes in our work.

Recent literature also examines various real-life mechanisms such as those from Boston \cite{APRS05}, Chicago \cite{CuJaLe06}, Milwaukee \cite{GrPeDu99}, \cite{Ro98}, and New York City \cite{APR05}.


\subsection{Notation, basic terms and our model}
\label{SS:Notation}

Let $I$ denote a nonempty set of students, and $S$ a nonempty set of schools. 
For all $s\in S$, we let $q_{s}$ denote the {\bf capacity} of $s$ and use the ordered tuple $Q = (q_s \vert s \in S)$ to encode all the capacities in a given problem involving the set $S$ of schools.

A {\bf preference profile} for a student $i \in I$, written ${P}_i$,
is a tuple $(S_{1},\dots,S_{n})$ where the $S_{j}$'s form a partition of $S$ and every element of $S_{j}$ is preferred to every element of $S_{k}$ if and only if $j<k$.
Define the {\bf ranking function} $\phi_i : S \rightarrow \N$  of a student $i \in I$ by letting $\phi_i(s)$ denote $i$'s {ranking} of $s \in S$. In other words $\phi_i(s) = j$ if $s \in S_j$. When each $S_{j}$ is singleton, we say that $i$'s preference profile is {\bf strict}, (in which case we can view $P_i$ as an $n$-vector).  If $s_{k},s_{l}\in S_{j}$ for some $j,k\neq l$, then we say that the student is {\bf indifferent} between $s_{k}$ and $s_{l}$. If $i$ prefers $s_{k}$ to $s_{l}$, we write $s_k \succ_i s_l$, or simply $s_k \succ s_l$ if $i$ is unambiguous.

A {\bf priority structure} for a school $s \in S$, written ${\Pi}_s$,
is a tuple $(I_{1},\dots,I_{n})$ where the $I_{j}$'s form a partition of $I$ and every element of $I_{j}$ is preferred to every element of $I_{k}$ if and only if $j < k$.

A {\bf school choice problem} for a set $S$ of schools and $I$ of students is a triple $({\bf P},{\bf \Pi}, Q)$, where ${\bf P} = \{P_i : i \in I\}$ is a set of preference profiles for the students in $I$, ${\bf \Pi} = \{\Pi_s : s \in S\}$ is a set of priority structures for the schools in $S$, and $Q$ encodes the capacities of schools in $S$. 

Given a school choice problem $({\bf P},{\bf \Pi}, Q)$ for a set $S$ of schools and $I$ of students, we define a {\bf matching} $M: I\to I\times S$ to be a function that associates every student with exactly one school, or potentially no school at all. 
We write $M_i = s$ if $M(i) = (i,s)$.

A matching $M'$ {\bf (Pareto) dominates} $M$ if $M'_i \succ_i M_i$ for all $i$ and $M'_j \succ_j M_j$ is strict for some $j$. A {\bf (Pareto) efficient matching} is a matching that is not (Pareto) dominated.

If $\mathfrak{M} = \mathfrak{M}({\bf P},{\bf \Pi}, Q)$ denotes the set of all matchings for the school choice problem $({\bf P},{\bf \Pi}, Q)$, then a {\bf matching mechanism} $\mathcal{M}$ is defined to be a function:
\[ \mathcal{M} : ({\bf P},{\bf \Pi}, Q) \mapsto \mathcal{M}({\bf P},{\bf \Pi}, Q) \]
that takes a school choice problem $({\bf P},{\bf \Pi}, Q)$ and produces a matching $\mathcal{M}({\bf P},{\bf \Pi}, Q) \in \mathfrak{M}({\bf P},{\bf \Pi}, Q)$.

A mechanism is \textbf{strategyproof} if no student can ever receive a more preferred school by submitting falsified, as opposed to truthful, preferences.


\section{Cardinal utility transformations and evaluation criteria \\for matching mechanisms}
\label{S:NewCriteria}

In this section, we use cardinal utility transformations to translate ordinal student preferences into cardinal ones and determine a total cost for any given assignment. Thus the school choice problem becomes a cost minimization problem. At that point, a combinatorial optimization algorithm can be invoked to find the optimal (lowest cost) matching (and we will do so in \S\ref{S:HA}). 

The question of what criteria to use to judge the quality or desirability of a mechanism is a difficult one; for example, see \cite{Mc09} where McFadden argues that tolerance of behavioral faults should be included in such a list of criteria. The goal of school districts when designing a school choice policy is not singular (unlike, for instance, the case of auction design where our sole objective is to maximize selling price). Thus, it is especially important to define feasible and meaningful yardsticks by which to measure the success of a given school choice mechanism. One could define the best school mechanism as one that minimizes the government education funding budgets, produces the most elite students, or improves the conditions of less-advantaged students the most, etc. The current literature on school choice uses stability, (Pareto) efficiency, and strategyproofness as the standard criteria for evaluating the desirability of a given mechanism. In our work, we emphasize student preferences. Obviously, the ultimate design depends on how we define the objectives 
of the school choice problem.

\subsection{Cardinal utility transformations}
\label{SS:CUT}

Let $I$ and $S$ be a set of students and schools, respectively, and let ${\bf P}$ be a set of preference profiles for the students in $I$. Let $\phi(S) \subset \N$ denote the set $\cup_{i \in I} \phi_i(S)$. Then a {\bf cardinal utility transformation} for $(I, S, {\bf P})$ is a strictly increasing function $f: \phi(S) \rightarrow \R$. 
We can use any strictly increasing function $f: \N \rightarrow \R$ but it suffices for $f$ to be defined only on $\cup_{i \in I} \phi_i(S)$.

It should be automatically clear that there exist infinitely many choices of $f$. Some of these can indicate specific utility and fairness assumptions. For instance a concave $f$ can be used to model risk-averse preferences while a convex $f$ can be used to reflect risk-loving preferences. In our analysis, we use two specific choices of $f$ to illustrate the application of our mechanism. 

We introduce a preference reverence index in \S\S\ref{SS:Index} and identify it as a type of cost to be minimized. This corresponds to picking a specific example of the simplest, linear, case of a cardinal utility transformation: Let $f$ be a linear transformation of the form $f(\phi(S))=a(\phi(S))+b$ where $a,b \in \R$. Such a choice of $f$ reflects the assumption that students possess uniform utility gaps between schools. If we are only given a list of ordinal preferences, one might invoke the principle of insufficient reason to justify such an assumption. However, given the often sharp differences in desirability between schools, this assumption may not be realistic.

One might alternatively try to choose $f$ in the spirit of philosopher John Rawls' Difference Principle. In the context of school choice, this might be interpreted as maximizing the utility of the worst-off student -- in other words, the student receiving their least-preferred school receives as highly a preferred school as possible. Inspection shows that a suitable choice of $f$ is the exponential function $f(\phi(S))=N^{(\phi(S))}$ where $N$ is the total number of students. Under this choice of $f$, we see that assigning all students their $N-1$ ranked school yields the same disutility as assigning one student their $N$ ranked school, thus stipulating that any maximization of net utility must necessarily give the student who received their least-preferred school as preferred school as possible. We define a notion of rank minimality in \S\S\ref{SS:Rank} with which we aim to capture this principle.

Of course, there exist other choices of $f$ that can be said to reflect other assumptions. Thus, in the class of mechanisms we consider, the mechanism designer chooses an $f$ to reflect the nature of the population as a whole, a preferred sense of fairness or a desired interpretation of collective utility.
 It helps to recall that the only constraints on $f$ are that:

\begin{enumerate}

\item $f$ respects the ordering of student preferences (i.e. $f$ is strictly increasing on $\phi(S)$), and:
\item the mechanism designer chooses a unique $f$ to be applied uniformly over all student preferences. 

\end{enumerate}

In order to make sure the algorithm we want to use works properly, we will also require that 

\begin{enumerate}
\item[(3)] the range of $f$ fall within the nonnegative numbers. 
\end{enumerate}


\subsection{A preference reverence index}
\label{SS:Index}

Let $I$ be a nonempty set of students, and $S$ be a nonempty set of $m$ schools. Recall that for any $i\in I$, $s\in S$, $\phi_i(s)$ is $i$'s ranking of $s$ and for any matching $M : I \rightarrow I \times S$, $M_i = s$ denotes that $M(i) = (i,s)$.
Let $\mathfrak{M}$ be the set of matchings.
Define $\mu: \mathfrak{M} \to \N$ by
\[
\mu(M)=\sum_{i \in I}\left ( \phi_i(M_i)-1 \right ) .\]
For any given $M \in \mathfrak{M}$ we will call $\mu(M)$ the {\bf preference reverence index} of $M$ or simply the {\bf preference index}.

Since $\mathfrak{M}$ is finite, $\mu(\mathfrak{M})$ is finite and hence there exists some $M\in \mathfrak{M}$ such that $\mu(M)\le \mu(M')$ for all $M'\in \mathfrak{M}$. We will describe a method of seeking and locating such a minimal index matching in \S\ref{S:HA}. In \cite{AACGKZZha10proc} we discussed several properties of this index; readers interested in other efficiency metrics might also like to see \cite{BK10}. Here we will only point out that using the index as the cost to be minimized in a school choice problem corresponds to using the function $f_1(n) = n-1$ as the cardinality transformation function.

The preference index measures how well ordinal preferences are being honored as a whole. Each time we move to the next-best choice in a student's ranking, this counts as ``1 violation" of their preferences, and we then add up the number of times we make such violations. Thus, perhaps a more apt title would be ``preference dismissal index" since it is a measure of how little the preferences are being ``honored" or ``revered." It should be noted that the preference index assumes that it is the same to give one student their fifth choice and one their first choice (Total=4) as it is to give two students their third choice (Total=4).


\subsection{Rank minimality}
\label{SS:Rank}

Let $\mathcal{S} = ({\bf P},{\bf \Pi}, Q)$ be a given school choice problem for a set $S$ of schools and a set $I$ of students. We 
define the {\bf rank} of a matching $M : I \rightarrow I \times S$, $M \in \mathfrak{M}({\bf P},{\bf \Pi}, Q)$, to be the maximal rank assigned to individual students under that matching:
\[ \rank M =  \max \{\phi_i(M_i) \vert i \in I\}. \]
We say that a matching $M : I \rightarrow I \times S$, $M \in \mathfrak{M}({\bf P},{\bf \Pi}, Q)$, is \textbf{rank-minimal} if it has minimal rank, or in other words if it minimizes the maximal individual assigned ranks in the following sense: 
\[ \max \{\phi_i(M_i) \vert i \in I\} \le \max \{\phi_i(M^{\prime}_i) \vert i \in I\} \quad \textmd{ for all } M^{\prime} \in \mathfrak{M}({\bf P},{\bf \Pi}, Q).\]
In words, this means that the worst off student under $M$ is better off than the worst off student under any other $M^{\prime}$.

Given the above definition, we will call a matching mechanism $\mathcal{M}$ {\bf rank-minimal} if for any set $S$ of schools and a set $I$ of students given, $\mathcal{M}$ maps any school choice problem $\mathcal{S} = ({\bf P},{\bf \Pi}, Q)$ for $S$ and $I$ to a rank-minimal matching.

Before moving forward, we compare our definitions here with a related notion, that of \emph{rank maximality} (cf.\ \cite[Def.1.2]{IKMMP06}):
A matching is {\bf rank maximal} if the maximum possible number of applicants are matched to their first choice, and subject to that condition, the maximum possible number of applicants are matched to their second choice, and so on.

Though this may sound similar to our notion of rank minimality, in many cases we will see there are some subtle differences. For instance consider the following preference profile for a school choice problem with five students and five schools, each with capacity 1:
\[
\begin{array}{cc} i_{1}: s_{1} \succ s_{2} \succ s_{3} \succ s_{4} \succ s_{5} \\
i_{2}: s_{2}  \succ  s_{3}  \succ  s_{4}  \succ  s_{5}  \succ  s_{1} \\
i_{3}: s_{3} \succ  s_{4}  \succ  s_{5}  \succ  s_{1}  \succ  s_{2} \\
i_{4}: s_{4}  \succ  s_{5}  \succ  s_{1}  \succ  s_{2}  \succ  s_{3} \\
i_{5}: s_{1}  \succ  s_{2}  \succ  s_{5}  \succ  s_{3}  \succ s_{4}\end{array}
 \]
There are two matchings which assign the most number of students (four) to their top choice:
\[ \begin{pmatrix}
i_{1} & i_{2} & i_{3} & i_{4} & i_{5} \\
s_{1} &  s_{2} & s_{3} & s_{4} & s_{5}
\end{pmatrix} \qquad \textmd{ and } \qquad \begin{pmatrix}
i_{1} & i_{2} & i_{3} & i_{4} & i_{5} \\
s_{5} &  s_{2} & s_{3} & s_{4} & s_{1}
\end{pmatrix}\]
It is easy to see that the only rank maximal matching is the first one, which has rank 3 (the second has rank 5). However if we want a rank minimal matching, we can find one with rank 2: 
\[ \begin{pmatrix}
i_{1} & i_{2} & i_{3} & i_{4} & i_{5} \\
s_{2} &  s_{3} & s_{4} & s_{5} & s_{1}
\end{pmatrix}\]

In \S\S\ref{S:HA}
we see that we can use an exponential cardinal utility transformation to ensure that an optimization algorithm can yield a rank-minimal matching for a given school choice problem.

%

\section{Cost-minimizing Mechanisms for the School Choice Problem}
\label{S:HA}

In \S\ref{S:NewCriteria} we introduced the notion of cardinal utility transformations and suggested two natural evaluation criteria for the school choice problem that correspond to two specific types of cardinal utility transformations. In this section we describe a flexible assignment mechanism which can be geared specifically toward these notions (or others, depending on the choice of cardinal utility transformation). 

The mechanism described here is built upon a combinatorial optimization algorithm known as the Hungarian algorithm. The Hungarian algorithm is traditionally used to find the minimum cost matching in various {\bf min-cost max-flow problems} such as assigning individuals to tasks or determining minimum cost networks in travel \cite{Ku55}, \cite{Ku56}. We note that the algorithm can be processed in polynomial time \cite{Mu57}, hence the mechanism itself can be effectively implemented via a computer program.

As the purpose of the Hungarian algorithm is to find the minimum cost matching, the first step in adapting the algorithm to the school choice problem is to define the cost of any particular matching. Here is where the cardinal utility transformation comes in. For a matching $M : I \rightarrow I \times S$, $M \in \mathfrak{M}({\bf P},{\bf \Pi}, Q)$, and a cardinal utility transformation $f$, we will define the {\bf cost} of $M$ to be: 
\[ C_f(M) = \sum_{i \in I} f(\phi_i(M_i)). \] 

Note here that choosing the cardinal utility transformation $f_1(n) = n-1$ ensures that the cost $C_{f_1}(M)$ of a matching $M$ is precisely the preference reverence index of $M$ (cf.\ \S\S\ref{SS:Index}). This measures the cost in terms of the number and extent of preference violations. Alternatively if we use the exponential cardinal utility transformation $f_2(\phi(S))=N^{(\phi(S))}$ where $N$ is the total number of students, then we will see that minimizing the resultant cost will ensure that the outcome matching will be rank-minimal. 

In the rest of this section we focus on various aspects of using the Hungarian algorithm in the school choice problem. We first describe the standard Hungarian algorithm for assignment problems with cost determined by a given cardinal utility transformation $f$ (\S\S\ref{SS:Description}). 
We then explain how we adapt it further to work for the school choice problem (\S\S\ref{SS:Modifications}). 
Next we study efficiency properties of this ``Hungarian" school choice mechanism (\S\S\ref{SS:Properties}) and how one can strategize under this mechanism (\S\S\ref{SS:HASA}).
We discuss some implementation issues in \S\S\ref{SS:ImplementationMultMin}.
%


\subsection{Description}
\label{SS:Description}

In the following we present an elementary description
of the Hungarian algorithm within the context of school choice. Our presentation is equivalent to the original development in \cite{Ku55}. 
For a more sophisticated discussion including computational complexity concerns and an exhaustive investigation of the many variants of the method that lead to impressive complexity improvements, see \cite[Ch.17]{Sc03}.

Let $I$ and $S$ be a set of students and schools, respectively, and assume that a student preference profile ${\bf P}$ is given. Also assume that we have selected a cardinal utility transformation $f$ and thus defined the associated cost function $C_f$. 
Since the space $\mathfrak{M}$ of all matchings is finite, $C_f(\mathfrak{M}) = \{C_f(M) : M \in \mathfrak{M}\}$ is finite and therefore there exists some $M\in \mathfrak{M}$ such that $C_f(M)\le C_f(M')$ for all $ M'\in \mathfrak{M}$. We would like to find such a minimal cost matching. 

Let $A =(a_{jk})$ be the $n\times m$ matrix such that $a_{jk}=\phi_{i_{j}}(s_k)$, encoding student preferences. 
Use the cardinal utility transformation $f$ on each of the entries to obtain a {\bf cost matrix} $\mathfrak{C}_f$; we would like this to have no negative entries, so it is useful to insist that the range of $f$ fall within the nonnegative numbers. 
For now assume that $n=m$, i.e., there is an equal number of students and schools and each school has a capacity of one.

For example for the following preference profile of three students for three schools:
\[
\begin{array}{ccc}
i_1: s_1 \succ s_2 \succ s_3\\
i_2: s_3 \succ s_2 \succ s_1 \\
i_3: s_2 \succ s_3 \succ s_1
\end{array}\]
the matrix $A$ of preferences would be:
\[\begin{array}{c| ccc} & s_{1} & s_{2} & s_{3}\\ \hline i_{1}& 1 & 2& 3\\ \hline i_{2}& 3 & 2 & 1\\ \hline i_{3} & 3 & 1 & 2\end{array}\]
and the associated cost matrix using $f_1$ would be:
\[ \mathfrak{C}_f = \begin{pmatrix}0&1&2\\
2&1&0\\
2&0&1
\end{pmatrix}.\]
Now the assignment problem reduces to: \emph{Given a cost matrix $\mathfrak{C}_f$, pick one entry from each row and each column such that the sum of the selected entries is minimal}. The Hungarian algorithm can then be used to find a solution to this reformulated problem. 

In this specific case the algorithm will run as follows (cf.\ \cite[Figure 6.1]{RoAn84}):
\begin{enumerate}
\item  Subtract the smallest entry in each row from each entry in that row. [After this stage, all rows have at least one zero entry, and all matrix entries are nonnegative.]
\item  Subtract the smallest entry in each column from each entry in that column. [After this stage, all rows and columns have at least one zero entry, and matrix entries are still nonnegative.]
\item Draw lines through appropriate rows and columns so that all the zero entries of the cost matrix are covered and the minimum number of such lines is used. [There may be several ways to do this, but the main point is that it can be done.]
\item \textbf{Test for optimality:} If the number of covering lines is $n$, then an optimal assignment of all zeroes is possible and we are done; the algorithm terminates. Otherwise, such an assignment is not yet possible, and we proceed to Step 5. 
\item Determine the smallest entry not covered by any line, subtract it from all uncovered entries and add it to all entries covered by both a horizontal and a vertical line. Return to Step 3.
\end{enumerate}

When the algorithm terminates at some reiteration of Step 4,  we use the position of the zeros in the terminal matrix to determine the desired assignment which corresponds to the least cost matching \cite{Mu57}. Here, for instance, is the outcome of the Hungarian algorithm for the preference profile above:
\[\begin{array}{c| ccc} & s_{1} & s_{2} & s_{3}\\ \hline i_{1}& \fbox{1} & 2& 3\\ \hline i_{2}& 3 & 2 & \fbox{1}\\ \hline i_{3} & 3 & \fbox{1} & 2\end{array}\]

We note that Step 5 crucially depends on the following 
\begin{theorem}[Theorem 6.1 \cite{RoAn84}]
If a number is added to or subtracted from all of the entries of any row or column of a cost matrix, then an optimal (minimum cost) assignment for the resulting cost matrix is also an optimal assignment for the original cost matrix.
\end{theorem}


\subsection{A ``Hungarian" school choice mechanism}
\label{SS:Modifications}

In adapting the Hungarian algorithm to the most general version of the school choice problem, we must make three key modifications, in order to accommodate 1) differing school capacities, 2) differing numbers of students, and 3) incomplete preference profiles. We consider these individually below.

The construction of the algorithm as we presented it above requires as input an $n \times n$ matrix of non-negative numbers,
and it selects as output a unique entry in each row and each column. We must modify the algorithm to accommodate school capacities, unequal numbers of seats and students, as well as preferences containing different numbers of ranked schools.\footnote{There are various reasons why students may choose to list different numbers of schools. For instance they may decide to pursue other options, such as private schooling, unless they happen to get into their top choice.} In the following we will call our modified mechanism the {\it Utility-Based Hungarian Mechanism} and denote it by $\mathcal{HM}_f$, where $f$ is the chosen cardinal utility transformation.

Assume columns represent schools and rows represent students in our matrix. To express school capacities, we simply add an extra column for each available seat at a school and enter the same preferences for that column.\footnote{Thus the matrix could have some repeated entries. In fact students could even submit non-strict rankings. In this manner the Hungarian Mechanisms introduced in this paper allow students to display indifferences between various schools with no penalty. We will say a bit more on this in Section \ref{S:Conclusion}.} Thus, each column would represent a seat at a school, rather than an entire school. Then, if there are an unequal number of available seats and students (i.e. an unequal number of rows and columns), we add dummy rows or dummy columns, which represent nonexistent students or schools. Thus, if a ``dummy row student'' were assigned to an actual school, this would signify an open seat at that school, whereas if an actual student were assigned a ``dummy column school'' this would signify that that student remains unassigned by the mechanism.

The third modification addresses the problem of families submitting incomplete preference profiles. Some school districts 
might not require that all preference profiles include the same number of schools, and it is likely that preference profiles would not be required to include all possible school assignments. 
Regardless, in order to run the Hungarian algorithm, it is necessary to devise a way of completing student preferences such that each student preference list assigns a rank to each school or seat.

A potential solution is to use dummy variables to complete any missing entries in the matrix. However, this method may invite students to strategize.
Even without complete information, students might be motivated to strategize by only submitting their first choice school, thereby weighting this choice with dummy variables so that the algorithm is more likely to select it.

Alternatively we can fill out the remainder of a student preference profile with an equal ranking for all unranked schools. More specifically if a student's preference profile contains only $r$ ranks, then we assign the rank $r+1$ to all the remaining schools. This incentivizes the completion of preference lists, since otherwise all remaining schools will be treated equally.  For instance, if a family puts only their first choice, all other choices will be considered ``second"; therefore they may get a school which they consider terrible at low cost as measured by the mechanism.  Thus it would behoove them to fill out as many schools as possible if they had a genuine preference for one over another.  

Now, let us focus on what happens for specific choices of cardinal utility transformations. 
Since our cost in the example from \S\S\ref{SS:Description} was precisely the preference reverence index itself, we can see that the resultant matching there has the smallest preference index with respect to each student's preferences.\footnote{Note that our use of the definite article for ``the smallest preference index'' is in fact not justified. The output of the Hungarian algorithm is not necessarily unique; there are cases with multiple minima to the cost function to be optimized. This is not an unresolvable issue however, and we address it in detail in \S\S\ref{SS:ImplementationMultMin}. Till then we will assume that in case of multiple minima, our mechanism will choose randomly between them.}  
 
Similarly if we use $f_2$ as the cardinal utility transformation, we will obtain a rank-minimal outcome. Recall that in this case the cost of the assignment will be given by
\[ C_{f_2} = \sum_{i \in I} f_2\left (\phi_i(\mathcal{HM}_{f_2}(i))\right ) = \sum_{i \in I} N^{\phi_i(\mathcal{HM}_{f_2}(i))},\]
where $N$ is the number of students. Now we assume, to reach a contradiction, that $\mathcal{HM}_{f_2}$ is not rank minimal and assigns some student $i$ to her $j$\textsuperscript{th} ranked school when there is indeed a way to assign all students to schools which they all want more than they want their $j$th choice. 
This implies that the cost term corresponding to the student $i$ will be $N^j$ for $\mathcal{HM}_{f_2}$, while for a rank-minimal matching, each student contributes a term to the cost a number that is less than that. In fact if there is a rank-minimal way to assign students to schools, say via the matching $M^{\prime}$, the corresponding $f_2$-cost will be less than $N^j$ as a whole:
\[ N^{j} \geq \sum_{i \in I} N^{ \phi_i(M^{\prime}(s)}. \]
Thus, if the Utility-Based Hungarian Mechanism were to assign one student her $j$\textsuperscript{th} ranked school when it was possible to assign all students to more preferred (ranked less than $j$) schools, this would contradict the fact that the Hungarian algorithm matching minimizes the cost given by the sum of the selected entries of the matrix.  Therefore, the outcome of $\mathcal{HM}_{f_2}$ has to be rank-minimal.


\subsection{Pareto efficiency and the Utility-Based Hungarian Mechanism}
\label{SS:Properties}
The Utility-Based Hungarian Mechanism is efficient:

\begin{theorem}
\label{T:Pareto}
 If the Utility-Based Hungarian Mechanism outputs matching $M$ under some monotonically increasing cardinal utility transformation function $f$, then $M$ is Pareto efficient. \end{theorem}

\begin{proof}
Assume, for the sake of contradiction, that $M$ is Pareto dominated by another matching $M'$. This necessarily means that two or more students prefer their matchings in $M'$ over $M$, while the matches for the rest of the students remain unchanged.  Since $f$ is strictly monotonically increasing, $M'$ must necessarily have a lower total ``cost" than $M$. But this contradicts our original hypothesis that the Utility-Based Hungarian Mechanism outputs $M$. 
\end{proof}

However, the converse is not necessarily true: 

\begin{theorem} 
\label{T:NotPar}
If $M$ is Pareto efficient, then there does not necessarily exist some monotonically increasing cardinal utility transformation function $f$ under which the Utility-Based Hungarian Mechanism yields $M$ as a solution. \end{theorem}

\begin{proof}
We prove with a counterexample. Consider the preference profile:
\[ \begin{array}{ccc} i_{1}: s_{1} \succ s_{2} \succ s_{3}\\i_{2}: s_{3} \succ s_{1} \succ s_{2}\\i_{3}: s_{3}\succ s_{2} \succ s_{1} \end{array}\]
Now, consider two Pareto efficient matchings associated with this preference profile:

\begin{center}
Matching 1\[ \begin{pmatrix}i_{1}& i_{2}& i_{3}\\ s_{1}&s_{2}&s_{3}\end{pmatrix}.\] Matching 2:\[\begin{pmatrix}i_{1}& i_{2}& i_{3}\\ s_{1}&s_{3}&s_{2}\end{pmatrix}.\]
\end{center}

Matching 1 has total ``cost" of $f(1)+f(3)+f(1)$ while Matching 2 has total cost of $f(1)+f(1)+f(2)$.  Since f is strictly monotonically increasing, $f(2)<f(3)$, thus the total cost of Matching 1 is strictly greater than that of Matching 2. Thus, under no $f$ will the Utility-Based Hungarian Mechanism ever choose Matching 1.  
\end{proof}

The two theorems above tell us that any choice of $f$ will yield Pareto efficient matchings under the Hungarian algorithm; however, not all Pareto efficient matchings can be found by the Hungarian algorithm under a suitable choice of $f$ . Thus, the set of all matchings output by the Hungarian algorithm under all $f$ defines a \emph{proper subset} of the set of Pareto efficient matchings.

\subsection{The Utility-Based Hungarian Mechanism and Strategic Action}
\label{SS:HASA}
Next we carefully examine the performance of $\mathcal{HM}_f$ with respect to strategic action. We first begin by describing how one can strategize under the given mechanism.

Given a choice of $f$, a student's preference profile over $n$ schools can be viewed simply as a permutation of $n$ numbers as represented in the cost matrix for the Hungarian algorithm. The set of all such $n!$ permutations constitutes the set of all possible preferences (and therefore available actions) to each student. Thus, under complete information, a general (albeit extremely inefficient) heuristic for determining when and how to strategize would be to run the Hungarian algorithm under each of these $n!$ permutations. If, under any of these ``falsified" preference permutations, a student receives a preferred school (or, if there are multiple solutions, a better expected outcome), then they should strategize by changing their true preference permutation to that falsified preference permutation.

Given that it is unlikely that students would possess complete information on the preferences of all of their classmates, and furthermore that assignment outcomes may be sensitive to minor changes in a classmate's preferences, what is perhaps more meaningful in applied contexts is whether there exists a simple strategy that students can apply under incomplete information. 

In answering this question, we begin with a brief example to illustrate that seemingly ``counterintuitive" strategies can actually be quite effective under the Hungarian algorithm. Recall the strategy for which the Boston Mechanism has most often been criticized (see for instance \cite{APRS05}): rather than ``squander" your first choice on a popular school that you would be unlikely to receive, rank some of the popular schools as less preferred while ranking the slightly less popular (and therefore more achievable) schools more highly.  If a student believes she has little chance of receiving a popular school, such a strategy might be adopted in hopes of securing a spot in a less popular school. It turns out that this is not a viable strategy in the Hungarian setting, but a totally opposite and a somewhat counterintuitive method will work. 

Consider the realistic scenario in which there are sharp discrepancies between schools in terms of desirability. In this case, a student with a notion of the relative popularity of each school among the general public might approximate the preferences of his classmates as more or less homogenous. In the example we construct, student $i_1$ believes that $s_1$ and $s_2$ are the two most popular schools, while $s_3$ and $s_4$ are the two least popular. Thus, $i_1$ believes that his classmates will, in general, have preferences: $i: s_1 \succ s_2 \succ s_3 \succ s_4$. Student $i_1$'s own truthful preferences differ slightly from those of the general public: $i_1: s_2 \succ s_3 \succ s_4 \succ s_1$. When viewed in a matrix, we have:
\[ \begin{array}{c| cccc} & s_{1} & s_{2} & s_{3} & s_{4}  \\ \hline i_{1}& f(4) & f(1) & f(2) & f(3) \\ \hline i_{2}& f(1) & f(2) & f(3) & f(4)  \\ \hline i_{3} & f(1) & f(2) & f(3) & f(4) \\ \hline \vdots & \vdots & \vdots & \vdots & \vdots  \\ \hline i_{n} & f(1) & f(2) & f(3) & f(4)  \end{array}\]
Here, for any choice of $f$, there are multiple solutions to the Hungarian algorithm. Inspection shows that under no cost-minimizing solution will $i_1$ receive $s_1$ and (assuming a solution is chosen randomly), he has equal probability of receiving $s_2$, $s_3$, or $s_4$ (because the algorithm eventually maps to zero all matrix entries on the first row except the first one). Thus, given a choice of $f$, $i_1$ has expected outcome: $\tfrac{f(1)+f(2)+f(3)}{3}$.
 
Now, assume $i_1$ attempts to strategize by submitting falsified preferences represented in the matrix:
$$\begin{array}{c| cccc} & s_{1} & s_{2} & s_{3} & s_{4}  \\ \hline i_{1}& f(2) & f(1) & f(3) & f(4) \\ \hline i_{2}& f(1) & f(2) & f(3) & f(4)  \\ \hline i_{3} & f(1) & f(2) & f(3) & f(4) \\ \hline \vdots & \vdots & \vdots & \vdots & \vdots  \\ \hline i_{n} & f(1) & f(2) & f(3) & f(4)  \end{array}$$
Under these falsified preferences, inspection shows that in any cost-minimizing solution $i_1$ receives $s_2$. Since $i_1$ is now receiving his first choice with certainty, and recalling that his true cardinal utility for $s_2$ is given by $f(1)$, we see that his expected outcome is now $f(1)$. Compare this to $i_1$'s expected outcome under truthful preferences: $f(1)\leq \frac{f(1)+f(2)+f(3)}{3}$. Since lower expected outcomes correspond to higher expected utility, we see $i_1$ is better off under these falsified preferences. As $i_1$ receives his first choice school with certainty here, no other set of falsified preferences can achieve a strictly better outcome for $i_1$, so $i_1$'s optimal strategy is to submit the falsified preferences above.

Note that $i_1$ improves his expected outcome by putting the most popular school $s_1$ \emph{higher} on his list, up from his fourth choice to his second, while also ranking the two least popular schools, $s_3$ and $s_4$, \emph{lower} on his list, pushing them back from his second and third choice to his third and fourth. 
Clearly, strategizing in the context of the Utility-Based Hungarian Mechanism is different and perhaps more subtle than in the Boston Mechanism. 

To develop a more explicit, nuanced, and reliable strategy, we again consider the scenario in which there are sharp discrepancies between the desirabilities of certain schools. We assume that, with some sense of the relative popularity of schools in mind, student $i_1$ ascribes homogenous preferences to his classmates. We begin with the following:

\begin{lemma}
\label{L:Permutation} 
If students $i_1, i_2, \cdots, i_N$ have homogenous preferences over $N$ schools, that is, for all~$s \in S$, $\phi_{i_k}(s)=\phi_{i_j}(s)$ for $1 \le j\neq k \le N$, then the Utility-Based-Hungarian Mechanism $\mathcal{HM}_f$ for the cardinal utility transformation $f$ finds $N!$ cost-minimizing solutions, each having a total ``cost" or sum of assigned cardinal utility values of $\sum_{k=1}^{N} f(\phi_{i_k}(s_k))$.  \end{lemma}  

\begin{proof}
We can represent the situation as follows:
$$\begin{array}{c| cccc} & s_{1} & s_{2} & \dots & s_N \\ \hline i_{1}& f (\phi_{i_1}(s_1)) & f(\phi_{i_1}(s_2)) &\dots & f(\phi_{i_1}(s_N)) \\ \hline i_{2}&  f(\phi_{i_2}(s_1)) & f(\phi_{i_2}(s_2))& \dots & f(\phi_{i_2}(s_N))  \\ \hline \vdots & \vdots & \vdots & \ddots & \vdots \\ \hline i_{N} & f(\phi_{i_N}(s_1)) & f(\phi_{i_N}(s_2)) & \dots & f(\phi_{i_N}(s_N))  \end{array}$$

However, for all $s \in S$ we have $\phi_{i_k}(s)=\phi_{i_j}(s)$ for $1 \le j\neq k \le N$, so we can rewrite:
$$\begin{array}{c| cccc} & s_{1} & s_{2} & \dots & s_N \\ \hline i_{1}& f (\phi_{i_1}(s_1)) & f(\phi_{i_1}(s_2)) &\dots & f(\phi_{i_1}(s_N)) \\ \hline i_{2}&  f(\phi_{i_1}(s_1)) & f(\phi_{i_1}(s_2))& \dots & f(\phi_{i_1}(s_N))  \\ \hline \vdots & \vdots & \vdots & \ddots & \vdots \\ \hline i_{N} & f(\phi_{i_1}(s_1)) & f(\phi_{i_1}(s_2)) & \dots & f(\phi_{i_1}(s_N))  \end{array}$$

Since the Hungarian algorithm requires each row and column to have a unique assignment, each possible cost-minimizing matching can be thought of as some permutation of the $N$ real numbers in the set $\{f(\phi_{i_1}(s_1)), f(\phi_{i_1}(s_2)), \dots, f(\phi_{i_1}(s_N))\} $. There are $N!$ such permutations, so there are $N!$ such cost-minimizing matchings. Furthermore, the sum of assigned cardinal utility values is simply the sum of each number in this permutation, and is therefore given by $\sum_{k=1}^{N} f(\phi_{i_k}(s_k))$. 
\end{proof}

A natural next step is:

\begin{lemma}
\label{L:OtherPermutations}
If student $i_1$ has preferences given by his ranking function $\phi_{i_1}$ over $N$ schools and his $N-1$ classmates have homogenous preferences given by the ranking function $\phi_{-i_1}$, then for each possible assignment of school $s_k$ to $i_1$, there exist $(N-1)!$ associated matchings, each with equal total ``cost" of $f(\phi_{i_1}(s_k))+\Big(\sum_{j=1}^N f(\phi_{-i_1}(s_j))\Big)-f(\phi_{-i_1}(s_k))$.
\end{lemma}

\begin{proof}
We can represent the situation as follows:
$$\begin{array}{c| cccc} & s_{1} & s_{2} & \dots & s_N \\ \hline i_{1}& f (\phi_{i_1}(s_1)) & f(\phi_{i_1}(s_2)) &\dots & f(\phi_{i_1}(s_N)) \\ \hline i_{2}&  f(\phi_{-i_1}(s_1)) & f(\phi_{-i_1}(s_2))& \dots & f(\phi_{-i_1}(s_N))  \\ \hline \vdots & \vdots & \vdots & \ddots & \vdots \\ \hline i_{N} & f(\phi_{-i_1}(s_1)) & f(\phi_{-i_1}(s_2)) & \dots & f(\phi_{-i_1}(s_N))  \end{array}$$
Assume student $i_1$ is assigned to school $s_k$. Then, since the Hungarian algorithm finds a unique assignment per row and per column, the rest of the students are matched in a ${(N-1) \times (N-1)}$ matrix. Therefore by Lemma \ref{L:Permutation}, there are $(N-1)!$ associated matchings for each assignment possibility  of $s_k$ to $i_1$. Furthermore, each of these matchings must have equal $f$-cost, since (again by Lemma \ref{L:Permutation}) the sum of the assigned cardinal utility values of these $(N-1)!$ classmates is~$\sum_{j=1}^N f(\phi_{-i_1}(s_j))-f(\phi_{-i_1}(s_k))$. Together with $i_1$'s cardinal utility value of $f(\phi_{i_1}(s_k))$, the total ``cost" of this matching is thus given by: 
\[ f(\phi_{i_1}(s_k))+\Big(\sum_{j=1}^N f(\phi_{-i_1}(s_j))\Big)-f(\phi_{-i_1}(s_k)),\] 
which completes the proof.\end{proof}

Now the next result follows immediately:

\begin{theorem} 
\label{T:Strategy}
If student $i_1$ has preferences given by his ranking function $\phi_{i_1}$ over $N$ schools and his $N-1$ classmates have homogenous preferences given by the ranking function $\phi_{-i_1}$, then under the Utility-Based Hungarian Mechanism $\mathcal{HM}_f$, $i_{1}$ will only receive school(s) $s_k$ where 
\[ f(\phi_{i_1}(s_k)) - f(\phi_{-i_{1}}(s_k)) \leq f(\phi_{i_1}(s_l)) - f(\phi_{-i_{1}}(s_l)) \textmd{ for all } l \neq k. \]
\end{theorem}

\begin{proof}
Assume, for the sake of contradiction, that $\mathcal{HM}_f$ assigned $i_1$ to some school $s_k$ while there existed some school $s_l$ such that
$$f(\phi_{i_1}(s_k)) - f(\phi_{-i_{1}}(s_k)) > f(\phi_{i_1}(s_l))- f(\phi_{-i_{1}}(s_l)).$$
Adding $\Big(\sum_{j=1}^N f(\phi_{-i_1}(s_j))\Big)$ to both sides, we get (by Lemma \ref{L:OtherPermutations}) that the total cost of a matching in which $i_1$ receives $s_k$ is strictly more than one in which $i_1$ receives $s_l$. This contradicts the fact that the Utility-Based Hungarian Mechanism is cost-minimizing.
\end{proof}

To illustrate how a strategizing student might apply Theorem \ref{T:Strategy} effectively to strategize, let us revisit the example with which we began our examination of strategic action, setting $f$ to be the identity function $f(\phi(S))=\phi(S)$:
\[ \begin{array}{c| cccc} & s_{1} & s_{2} & s_{3} & s_{4}  \\ \hline i_{1}& 4 & 1 & 2 & 3 \\ \hline i_{2}& 1 & 2 & 3 & 4  \\ \hline i_{3} & 1 & 2 & 3 & 4 \\ \hline \vdots & \vdots & \vdots & \vdots & \vdots  \\ \hline i_{n} & 1 & 2 & 3 & 4  \end{array}\]
 Applying Theorem \ref{T:Strategy}, $i_1$ would first compute:
\begin{eqnarray*}
f(\phi_{i_1}(s_1))-f(\phi_{-i_1}(s_1))&=&(4-1)~=+3, \\ 
f(\phi_{i_1}(s_2))-f(\phi_{-i_1}(s_2))&=&(1-2)~=-1, \\ 
f(\phi_{i_1}(s_3))-f(\phi_{-i_1}(s_3))&=&(2-3)~=-1, \\ 
f(\phi_{i_1}(s_4))-f(\phi_{-i_1}(s_4))&=&(3-4)~=-1,
\end{eqnarray*}
and see that under no cost-minimizing matching can he receive $s_1$, and that he is equally likely to receive $s_2$, $s_3$, and $s_4$. In order to ensure receiving his first choice $s_2$, student $i_1$ must submit falsified preferences such that $f(\phi_{i_1}(s_2))-f(\phi_{-i_1}(s_2)) \leq f(\phi_{i_1}(s))-f(\phi_{-i_1}(s))$ for all $s \in S$. Indeed falsifying his preferences as $i_1: s_2 \succ s_1 \succ s_3 \succ s_4$ works. Under these falsified preferences, we have
\begin{eqnarray*}
f(\phi_{i_1}(s_1))-f(\phi_{-i_1}(s_1))&=&(2-1)=+1,\\
f(\phi_{i_1}(s_2))-f(\phi_{-i_1}(s_2))&=&(1-2)=-1,\\
f(\phi_{i_1}(s_3))-f(\phi_{-i_1}(s_3))&=&(3-3)=0,\\
f(\phi_{i_1}(s_3))-f(\phi_{-i_1}(s_3))&=&(4-4)=0.
\end{eqnarray*}

Thus we conclude that the Utility-Based Hungarian Mechanism is not immune to strategic action, even under incomplete information. In realistic scenarios, preferences will never be completely homogenous as we assumed. One can argue that the above strategy will become less reliable the more ``heterogenous" preferences become; the authors have not explored this direction.


\subsection{An implementation issue: multiple minima}
\label{SS:ImplementationMultMin}

In some instances the cost function we define might not correspond to a strict ordering. For instance with the cardinal utility transformation $f_1$, the cost function we obtain (the preference reverence index) may induce a non-strict ordering of the possible matchings, and a given preference profile might have multiple minimum preference index solutions. 
For example, consider the following preference profile:
\[
\begin{array}{cc} i_{1}:s_{1} \succ s_{2} \succ s_{3} \succ s_{4}\\ i_{2}: s_{4} \succ s_{2} \succ s_{1} \succ s_{3}\\ i_{3}:s_{3} \succ s_{1} \succ s_{4} \succ s_{2}\\ i_{4}:s_{3} \succ s_{4} \succ s_{2} \succ s_{1}\end{array}
\]
Here, 
there are three minimum cost (minimum preference index) matchings:
\[ \text{ Matching }\#1 \quad (C_{f_1}=2): \qquad
\begin{pmatrix}i_{1}& i_{2}&i_{3}&i_{4}\\ s_{1}&s_{2}&s_{3}&s_{4}\end{pmatrix}
\]
\[ \text{ Matching }\#2 \quad (C_{f_1}=2): \qquad
\begin{pmatrix}i_{1}& i_{2}&i_{3}&i_{4}\\ s_{2}&s_{4}&s_{1}&s_{3}\end{pmatrix}
\]
\[ \text{ Matching }\#3 \quad (C_{f_1}=2): \qquad
\begin{pmatrix}i_{1}& i_{2}&i_{3}&i_{4}\\ s_{1}&s_{4}&s_{3}&s_{2}\end{pmatrix}\]
The Utility-Based Hungarian Mechanism $\mathcal{HM}_{f_1}$ as defined above will output Matching 1. Is this a desirable situation?

The underlying theoretical problem of finding all possible minimum cost assignments by the Hungarian algorithm was addressed in \cite{FuTo92} (see \cite{Fu94} for an improvement on the main (polynomial time) algorithm used in \cite{FuTo92} and \cite{MaPl05} for more recent work in a similar vein).  Thus it is possible to find all minimum cost (minimum preference index) solutions using a mechanism adopting the Hungarian algorithm. This in turn raises the natural question:
How does one choose among multiple minima? We propose two possible approaches to deal with this issue.

\begin{enumerate}
\item If one intends to promote fairness by narrowing the discrepancies between the rankings of student assignments, then the matching with the minimum variance across individual student preference indices should be chosen. 
\item  If one intends to manage priorities and choose ``the most stable" matching, then the matching with the fewest number of students whose priority has been violated should be chosen.\footnote{This is not the same as looking at the total number of priority violations since a student could have his priority violated by multiple students. Once a student's priority has been violated, he can pursue legal action whether his priority is violated by one or by one hundred students. Thus, policy makers will probably be more concerned with how many students had their priorities violated rather than how many total priority violations there are.}
\end{enumerate}

Of course, one may use both of these in succession.

This incidentally addresses a possible concern about the Hungarian algorithm: its dependence on the order of the rows and the columns of the input matrix. Especially when there are multiple minimal index solutions, the order in which students or schools are listed may indeed affect the outcome, and the output matching may be different in different cases (though any two outcomes in such a scenario will have the same minimal cost). However if we modify our mechanism to look instead for all possible minimum cost matchings, this no longer creates a problem.
Thus, the order of the rows or columns ultimately does not matter because: (1) If there is a unique cost-minimizing solution, the order does not affect the outcome; and (2) if there are multiple cost-minimizing solutions, we can  find all of them using our mechanism, with adaptations a la \cite{FuTo92}.

The situation is somewhat different in the case of $\mathcal{HM}_{f_2}$ where the outcome matching is rank-minimal. More specifically, if we were to use \cite{Fu94} to find all minimum cost matchings with respect to the cardinal utility transformation $f_2$, we would not necessarily find all rank-minimal matchings. This is because there might exist two rank-minimal matchings for a given preference profile that have different costs. 
For example, consider the following preference profile:
\[
\begin{array}{cc} i_1: s_3  \succ  s_2  \succ  s_1 \\
i_2: s_2  \succ  s_3  \succ  s_1\\
i_3: s_2  \succ  s_1  \succ  s_3 \end{array}
\]
There are two matchings which both have the minimum rank-2 but have different costs:
\[ \text{ Matching 1 (Rank-2, } C_{f_2}=15): 
\begin{pmatrix}
i_1 & i_2 & i_3 \\ s_3 & s_2 & s_1
\end{pmatrix}
\]
\[ \text{ Matching 2 (Rank-2, } C_{f_2}=27): 
\begin{pmatrix}
i_1 & i_2 & i_3 \\ s_2 & s_3 & s_1
\end{pmatrix}
\]
Here, Matching 1 is the minimum $f_2$-cost matching that will be found by the Utility-Based Hungarian Mechanism when we use $f_2$ as the cardinal utility transformation, while Matching 2 is another matching with the same minimal rank. Notice that Matching 1 Pareto dominates Matching 2, so we can see that at least with regard to one other criterion, Matching 1 is measurably better.\footnote{In fact, all minimum $f_2$-cost matchings are not only rank-minimal, but are also Pareto efficient (cf.\ Theorem \ref{T:Pareto}), though
not all rank-minimal matchings will be Pareto efficient.}
One can nonetheless see a heuristic method to find all rank-minimal matchings: {\it List all possible matchings in order of increasing $f_2$-cost. Every matching listed above the first one that changes rank will be rank-minimal.}

In case of multiple minima for $f_2$-cost, we can again use \cite{Fu94} to find all minimum $f_2$-cost matchings, and employ analogues of the two aforementioned tie-breaking criteria to determine which matching to pick. We can also use cost functions like $C_{f_1}$ as additional criteria.\footnote{In the example given, $C_{f_1}=1$ for Matching 1 and $C_{f_1}=3$ for Matching 2.}


\section{Conclusion}
\label{S:Conclusion}

Current school choice mechanisms focus on balancing student preferences and school priorities, and the resulting matches sacrifice desirable characteristics.Ê Since a good public education is a scarce resource, there is no way to assign students to schools in such a way that all students attend top schools.Ê In our approach we chose to focus exclusively on student preferences.\footnote{Our focus on student preferences over school priorities is natural in the current climate in which the public debate over charter schools and school vouchers rages in an attempt to offer parents more control over their children's educational choices. 
See for instance several recent feature-length movies on school choice: Waiting for ``Superman", The Lottery, The Cartel Movie.}  As a result school choice became a one-sided matching problem. We next used the notion of a cardinal utility transformation to convert student preferences into cardinal utility ranks and thus translated the school choice problem into a cost minimization problem, where the cost depends on the choice of cardinal utility transformation $f$. Two particular instances of $f$ corresponded neatly to two natural criteria frequently used in combinatorial optimization problems.

The mechanism presented here was adapted from the Hungarian algorithm \cite{Ku55} which was developed as a combinatorial solution to the assignment problem.Ê Our modifications included a re-interpretation of assignments taking into account school capacities and required that we be allowed to ``complete'' submitted student preference profiles. 
With the introduction of this flexibility came the requirement that we determine a fair 
 way of completing student preference profiles.Ê In the profile completion process we sought to avoid confounding the problem of having non-participatory parents/adults costing unknowing and often powerless children a seat at the best possible school.Ê

An obvious weakness of our proposed mechanism is instability.Ê Since we ignored priorities as a whole, it was natural that the outcomes would suffer in terms of stability. We see a robust incorporation of school priories as an interesting direction for further investigation.

Another interesting direction for future work is in the incorporation of indifferences.Ê In particular, the Utility-Based Hungarian Mechanism affords students the opportunity to express indifferences. If a student is indifferent between several schools, the ranking number for these schools is simply repeated in the matrix. It seems on a cursory inspection that a dishonest representation of indifferences can only serve to harm a student's chance of receiving his preferred schools. There has been much work focusing on indifferences in school priorities (see for instance \cite{ErEr08}), but not as much has been done on student indifferences. We believe that this is an interesting thread to follow.

\section*{Acknowledgments}
This paper evolved from work the authors presented in a special paper session on computational social choice at a conference (ISAIM 2012 (International Symposium on Artificial Intelligence and Mathematics (ISAIM 2012), Fort Lauderdale, Florida, USA, January 9-11, 2012). A brief paper \cite{AACGKZZha10proc} was published as part of the proceedings of that conference. The current paper refers to some results in those proceedings, but is substantially revised and independent. Thus we hereby submit it for publication; the different title indicates the substantially different focus of the current paper.


\bibliographystyle{siam}   

\bibliography{SchoolChoiceBibliography}


\end{document}